\documentclass[10pt]{article}

\usepackage{amsmath, amsthm, amssymb, enumerate}
\usepackage{color}
\usepackage{datetime}
\usepackage{fancyhdr}

\chardef\coloryes=0 
\chardef\isitdraft=0 

\ifnum\isitdraft=1 
   \def\eqref#1{({\ref{#1}})}                

\usepackage[color,notcite]{showkeys}

\usepackage[notcite]{showkeys}

\definecolor{labelkey}{gray}{.3}

\definecolor{refkey}{rgb}{.3,0.3,0.3}

\else

  \def\startnewsection#1#2{\section{#1}\label{#2}\setcounter{equation}{0}}   
  \def\nnewpage{} 
\fi

\begin{document}
\def\ques{{\colr \underline{??????}\colb}}
\def\nto#1{{\colC \footnote{\em \colC #1}}}
\def\fractext#1#2{{#1}/{#2}}
\def\fracsm#1#2{{\textstyle{\frac{#1}{#2}}}}   
\def\nnonumber{}


\def\colr{{}}
\def\colg{{}}
\def\colb{{}}
\def\cole{{}}
\def\colA{{}}
\def\colB{{}}
\def\colC{{}}
\def\colD{{}}
\def\colE{{}}
\def\colF{{}}

\ifnum\coloryes=1

  \definecolor{coloraaaa}{rgb}{0.1,0.2,0.8}
  \definecolor{colorbbbb}{rgb}{0.1,0.7,0.1}
  \definecolor{colorcccc}{rgb}{0.8,0.3,0.9}
  \definecolor{colordddd}{rgb}{0.0,.5,0.0}
  \definecolor{coloreeee}{rgb}{0.8,0.3,0.9}
  \definecolor{colorffff}{rgb}{0.8,0.3,0.9}
  \definecolor{colorgggg}{rgb}{0.5,0.0,0.4}

 \def\colg{\color{colordddd}}
 \def\colb{\color{black}}
 \def\colr{\color{red}}
 \def\cole{\color{colorgggg}}

 \def\colA{\color{coloraaaa}}
 \def\colB{\color{colorbbbb}}
 \def\colC{\color{colorcccc}}
 \def\colD{\color{colordddd}}
 \def\colE{\color{coloreeee}}
 \def\colF{\color{colorffff}}
 \def\colG{\color{colorgggg}}

\fi
\ifnum\isitdraft=1
   \chardef\coloryes=1 
   \baselineskip=17pt
   \input macros.tex
   \def\blackdot{{\color{red}{\hskip-.0truecm\rule[-1mm]{4mm}{4mm}\hskip.2truecm}}\hskip-.3truecm}
   \def\bdot{{\colC {\hskip-.0truecm\rule[-1mm]{4mm}{4mm}\hskip.2truecm}}\hskip-.3truecm}
   \def\purpledot{{\colA{\rule[0mm]{4mm}{4mm}}\colb}}
   \def\pdot{\purpledot}
\else
   \baselineskip=15pt
   \def\blackdot{{\rule[-3mm]{8mm}{8mm}}}
   \def\purpledot{{\rule[-3mm]{8mm}{8mm}}}
   \def\pdot{}
\fi

\def\tdot{\fbox{\fbox{\bf\tiny I'm here; \today \ \currenttime}}}
\def\nts#1{{\hbox{\bf ~#1~}}} 
\def\nts#1{{\colr\hbox{\bf ~#1~}}} 
\def\ntsf#1{\footnote{\hbox{\bf ~#1~}}} 
\def\ntsf#1{\footnote{\colr\hbox{\bf ~#1~}}} 
\def\bigline#1{~\\\hskip2truecm~~~~{#1}{#1}{#1}{#1}{#1}{#1}{#1}{#1}{#1}{#1}{#1}{#1}{#1}{#1}{#1}{#1}{#1}{#1}{#1}{#1}{#1}\\}
\def\biglineb{\bigline{$\downarrow\,$ $\downarrow\,$}}
\def\biglinem{\bigline{---}}
\def\biglinee{\bigline{$\uparrow\,$ $\uparrow\,$}}

\def\tilde{\widetilde}

\newtheorem{Theorem}{Theorem}[section]
\newtheorem{Corollary}[Theorem]{Corollary}
\newtheorem{Proposition}[Theorem]{Proposition}
\newtheorem{Lemma}[Theorem]{Lemma}
\newtheorem{Remark}[Theorem]{Remark}
\newtheorem{definition}{Definition}[section]
\def\theequation{\thesection.\arabic{equation}}
\def\endproof{\hfill$\Box$\\}
\def\square{\hfill$\Box$\\}
\def\comma{ {\rm ,\qquad{}} }            
\def\commaone{ {\rm ,\qquad{}} }         
\def\dist{\mathop{\rm dist}\nolimits}    
\def\sgn{\mathop{\rm sgn\,}\nolimits}    
\def\Tr{\mathop{\rm Tr}\nolimits}    
\def\div{\mathop{\rm div}\nolimits}    
\def\supp{\mathop{\rm supp}\nolimits}    
\def\divtwo{\mathop{{\rm div}_2\,}\nolimits}    
\def\re{\mathop{\rm {\mathbb R}e}\nolimits}    

\def\indeq{\qquad{}}                     
\def\period{.}                           
\def\semicolon{\,;}                      

\title{A local asymptotic expansion for a local solution of the Stokes system}
\author{G\"{u}her \c{C}aml{\i}yurt and Igor Kukavica}
\maketitle

\date{}


\medskip

\indent Department of Mathematics, University of Southern California, Los Angeles, CA 90089\\
\indent e-mails: camliyur@usc.edu, kukavica@usc.edu

\begin{abstract}
We consider solutions of the Stokes system in a neighborhood of a
point in which the velocity $u$ vanishes of order $d$. We prove that there
exists a divergence-free polynomial $P$ in $x$ with $t$-dependent coefficients 
of degree $d$ which approximates the solution $u$ of order $d+\alpha$
for certain $\alpha>0$.
The polynomial $P$ satisfies a Stokes equation with a forcing term which
is a sum of two polynomials in $x$ of degrees $d-1$ and $d$.
The results extend to Oseen systems and to the Navier-Stokes equation.
\end{abstract}



\startnewsection{Introduction}{sec1} 
In this paper, we study local asymptotic development of solutions of
local solutions for the Stokes equation in the unit cylinder.
Namely, given $f = (f_1 (x,t), f_2(x,t), \ldots, f_n(x,t))$ we seek 
a polynomial in $x$ which approximates a solution 
$
   u = (u_1(x,t), u_2 (x,t), \ldots, u_n (x,t))   
$
of the system
  \begin{align}
   & u_t - \triangle u + \nabla p = f, 
    \label{EQ11}\\
  & \nabla \cdot u =0,
   \label{EQ12}
  \end{align}
around a point where the solution vanishes
of order $d$. The solution is not assumed to
have a high degree of regularity and thus the Taylor expansion
is not available.
Replacing the force with a matrix of functions in the divergence form
we also obtain development for 
solutions of the Navier-Stokes equations around a vanishing point as a consequence.

Fabre and Lebeau in \cite{FL1,FL2} showed that the system
\eqref{EQ11}--\eqref{EQ12} has a unique continuation property,
i.e., local solutions of \eqref{EQ11}--\eqref{EQ12}
can not vanish to infinite order unless they vanish identically.
Having a priori estimates on solutions with respect to their vanishing
order is considered a crucial step in many applications. 
For instance, using a priori estimates on asymptotic polynomials Han
\cite{H2} improves the classical Schauder estimates in a way 
that the estimates of solutions and their derivatives at one point
depend on the coefficient
and the nonhomogeneous terms at that particular point. 
Also, 
Hardt and Simon \cite{HS} applied an estimate of Donnelly and Fefferman for  
the order of vanishing of eigenfunctions to find an
asymptotic bound of the $(n-1)$-dimensional measure 
of $v_j^{-1} \{0\}$, where $v_j$
is an eigenfunction 
corresponding to the $j$-th eigenvalue of the Laplacian on a compact
Riemannian manifold.
  
The method we use in proving the main theorem was introduced by
Q.~Han, who  in \cite{H2} found an asymptotic 
development of a solution of a parabolic equation
of an arbitrary degree (cf.~also \cite{H1} for the elliptic case).
The main idea in \cite{H2} is based on a local expansion
of the corresponding fundamental solution of the global
linear equation.

There are several key difficulties when trying to extend the results to
the Stokes equation \eqref{EQ11}--\eqref{EQ12}.
First,
due to  presence of the pressure, it is not reasonable
to expect that the velocity and the pressure would vanish at the same
point
(for instance, the unique continuation result of Fabre and Lebeau
gives a unique continuation property for $u$ and not for the pair
$(u,p)$).
Thus in our main result we do require $p$ to vanish.
The second difficulty is the lack of smoothing in the time variable
in the system, which is
a well-known problem for local solutions of the Stokes and Navier-Stokes
systems.
Indeed, taking the divergence of the evolution equation for the
velocity gives
   \begin{align}
    & \triangle p = \nabla \cdot f\period 
  \label{EQ14}
  \end{align}  
which does not contain any smoothing in the time variable.
The third difficulty is the nonlocal nature of the Stokes kernel,
which in particular causes the Stokes kernel to decay polynomially,
rather than exponentially as it is the case for the scalar equations.

We note here that there have been many works on unique continuation of
elliptic and parabolic equations 
showing that, under various assumptions on coefficients,
no solution can vanish to infinite order
(cf.~\cite{AE,AMRV,CRV,DF,EFV,EV,GL,JK,KT,SS1,SS2} for instance);
for more complete reviews, see \cite{K1,K2,V}.
Unique continuation questions for the Stokes and Navier-Stokes
systems were addressed in \cite{CK,FL1,FL2,Ku}.

The paper is organized as follows. In Section~\ref{sec2}, we state the
main results, Theorem~\ref{T01} and~\ref{T02}, addressing
the forces in standard and divergence forms respectively.
We also state the two corollaries concerning the Navier-Stokes and
Oseen systems. In Section~\ref{sec3} we recall the properties of the
Stokes kernel, while the last part contains a construction of a
particular solution vanishing of order $d$ as well as the proof of
Theorem~\ref{T01}. 


%
%

\startnewsection{Notation and the main result on the asymptotic expansion }{sec2}
In this paper, we consider a solution $(u,p)$ of the Stokes system
\eqref{EQ11}--\eqref{EQ12} in an open set containing $(0,0)$ (which can
always be assumed using translation).
For any $(x,t) \in \mathbb{R}^n \times \mathbb{R}$ and $r>0$ we denote
the parabolic cylinder
label by $(x,t)$ with radius $r>0$ by
  \begin{align*}
   & Q_r (x,t) = \bigl\{ (y,s) \in \mathbb{R}^n \times \mathbb{R}: 
         |y-x|<r, -r^2 < s-t<0 
                 \bigr\}\period
  \end{align*}
The corresponding parabolic 
norm for $(x,t) \in \mathbb{R}^n \times \mathbb{R}$ is given by 
  \begin{align*}
   & |(x,t)|= (|x|^2 + |t|)^{1/2}\period
  \end{align*}
Denote by $W^{m,1}_{q} (Q_1)$  the Sobolev space of 
$L^{q}(Q_{1})$
functions whose all the $x$-derivatives up to 
$m$-th order and $t$-derivative of first order belong to $L^q(Q_1)$.

\cole
\begin{Theorem}
\label{T01}
Let $q> 1+ n/2$. 
Suppose that $f_j \in L^q (Q_1)$, for $j = 1,2, \ldots,n$ satisfy
  \begin{equation}
  \|f_j\|_{L^q (Q_r)} \leq \gamma r^{d-2+\alpha + {n+2}/ q} 
   \comma r\leq 1
   \label{EQ15}
  \end{equation}
for some constants $\gamma>0$ and $\alpha \in (0,1)$, where $d\geq 2$
is an integer. Then for any solution 
$u= (u_1, \ldots, u_n)\in W_{q}^{2,1}(Q_{1})$
of \eqref{EQ11}--\eqref{EQ12}
there exists $P= (P_{d,t}^1 , \ldots, P_{d,t}^n)$, 
whose each component $P_{d,t}^j$ is a polynomial 
in $x$ of degree less than or equal to $d$,
such that
  \begin{align}
   |u_j (x,t) - P_{d,t}^j (x,t)| \leq C \left( \gamma + \sum_{k=1}^n \| u_k\|_{W^{2,1}_q (Q_1)} \right) |(x,t)|^{d+ \alpha}
   \comma (x,t) \in Q_{1/2},
  \end{align}
where $C$ is a positive constant depending on $n$, $q$, $d$, and $\alpha$.
Moreover, $P$ satisfies the Stokes system
  \begin{align}
   &\partial_t P^j_{d,t} (x) - \triangle P^j_{d,t} (x)  +  \partial_j R (x)
    =\sum_{i=d-1}^{d} \sum_{|\alpha|=i} C_{\alpha, t}^{j} \frac{x^{\alpha}}{\alpha !}   
   \comma j=1,\ldots,d
  \label{EQ16} 
  \\\indeq &
  \nabla\cdot P  = 0, 
  \label{EQ01}
  \end{align}
where $R$ is the corresponding pressure. 
\end{Theorem}
\colb

\begin{Remark}
{\rm
The pressure term is found explicitly in the 
proof of the main theorem and is given by 
  \begin{align*}
   R(x,t) 
    = \sum_{i=0}^{d-2} \sum_{|\alpha| =i} D^{\alpha}_x (p- \triangle^{-1} \nabla \cdot f)
  \end{align*}
for $(x,t)\in Q_{1}$.
}
\square
\end{Remark}

As we are interested in obtaining estimates in $Q_1$, we assume 
without loss of generality that
  \begin{equation}
   f(x,t) = 0
   \comma |(x,t)|\ge1
   \period
   \label{EQ37}
  \end{equation}

In the case when the function on 
the right side of \eqref{EQ11} is in the divergence form, the Stokes system reads as
  \begin{align}
   &
   \partial_t u_k - \triangle u_k + \partial_k p = \partial_j g_{jk}
   \comma k=1,\ldots, n 
   \label{EQ08}
   \\ & 
   \nabla \cdot u =0, 
   \label{EQ02}
  \end{align}
for some function 
$g = (g_{jk})_{j,k=1}^{n} \in W_q^{1,0} (Q_1)$. 
Here also we may assume without loss of generality that
  \begin{equation}
   g(x,t) = 0
   \comma |(x,t)|>1
   \period
   \label{EQ51}
  \end{equation}
Then we have the following variant of Theorem~\ref{T01}.

\cole
\begin{Theorem} 
\label{T02}
Assume that $q> 1+ n/2$. 
Let $g= [g_{jk}]\in W_{q}^{2,1}$ be an $n\times n$ matrix of functions
that satisfies
  \begin{align}
   |g_{jk} (x,t)| \leq \gamma |(x,t)|^{d-1+\alpha}
   \comma  |(x,t)|<1,
   \commaone j,k=1,\ldots,n
  \end{align}  
for some constants $\gamma >0$ and $\alpha \in (0,1)$. 
Then for any solution $u = (u_1, u_2, \ldots, u_n)\in W_{q}^{2,1}$ of 
\eqref{EQ08}--\eqref{EQ02} there exists $P= (P_{d,t}^1 , \ldots, P_{d,t}^n)$ whose each component $P_{d,t}^j$ is a polynomial in $x$ of degree less than or equal to $d$ such that
  \begin{align}
   |u_j (x,t) - P_{d,t}^j (x,t)| 
       \leq 
        C \left(
             \gamma + \sum_{k=1}^n \| u_k\|_{W^{2,1}_q (Q_1)}
          \right) |(x,t)|^{d+ \alpha}, 
  \end{align}
for any $(x,t) \in Q_{1/2}$, where $C$ is a positive constant depending on $n$, $q$, $d$, $\alpha$. Also, $P$ satisfies the Stokes system
  \begin{align}
       &\partial_t P_{d,t}^j - \triangle P_{d,t}^j  + \partial_j R (x) 
          = \sum_{i=d-1}^{d} \sum_{|\alpha|=i} C_{\alpha, t} \frac{x^{\alpha}}{\alpha !}   
   \comma j=1,\ldots,n
   \label{EQ03}
  \\\indeq &
    \nabla\cdot P  = 0  
   \label{EQ04}
\end{align}  
where $R$ is the corresponding pressure.
\end{Theorem}
\colb

Having a force in 
divergence form on the right side of \eqref{EQ02} allows us to apply the above results to the solutions of the Navier-Stokes equations.

\cole
\begin{Corollary} 
\label{C01}
Let $q> 1+ n/2$. Suppose  that
$u= (u_1, \ldots, u_n) \in W^{2,1}_q(Q_1)$
solves the Navier-Stokes equations
  \begin{align}
   & \partial_t u 
     - \triangle u 
     + \nabla(u \otimes u)  
     + \nabla p= 0,
   \label{EQ05}
 \\
& \nabla \cdot u =0
   \period
   \label{EQ06}
\end{align}
Also, assume that $u$ vanishes of the order at least $d\geq 2$. Then there exists $P= (P^1_{d,t}, \ldots, P^n_{d,t})$ whose each component $P^j_{d,t}$ is a polynomial in $x$ of degree less than or equal to $d$ such that 
  \begin{align}
   |u_j (x,t)- P^j_{d,t} (x)| \leq C |(x,t)|^{d+1}\comma (x,t) \in Q_{1/2},
   \label{EQ19}
\end{align}
for $j=1,\ldots, n$, where $C$ is a positive constant depending on $n$, $d$, $q$, and $u$. Moreover, $P$ satisfies the Stokes system
  \begin{align}
   &\partial_t P^j_d - \triangle P^j_d   + \partial_j R (x) = \sum_{i=d-1}^{d} \sum_{|\alpha|=i} C_{\alpha, t} \frac{x^{\alpha}}{\alpha !} 
   \comma j=1,\ldots,n
   \label{EQ20}
    \\&\nabla\cdot P  =0,
   \label{EQ28}
\end{align}
where $R$ a suitable pressure term depending on $u$ and $p$. 
\end{Corollary}
\colb

We note that $u$ is not assumed to be smooth in the space or time variable. 
Therefore, the inequality in
\eqref{EQ19} can not be obtained by expanding the solution in the Taylor series. 

The result in Theorem~\ref{T01} can also be applied to the Oseen
system considered
in \cite{FL1}. 
\cole
\begin{Corollary} 
\label{C02}
Let $q> 1 + n/2.$ Suppose  that
$u = (u_1, \ldots, u_n) \in W^{2,1} _q (Q_1)$ solves the Oseen system
\begin{align}
& \partial_t u - \triangle u + (a.\nabla) u + \nabla p  =0  
\comma j= 1, \ldots,n  \label{EQ52} \\
& \nabla \cdot u =0 
\period
\label{EQ53}
\end{align}
where $a= (a_1, \ldots, a_n)\in L^{\infty} (Q_1)$.  
Also, assume that $u$ vanishes of the order at least $d \geq 2$. 
Then there exists $P= (P^1_{d,t}, \ldots, P^n_{d,t})$ 
whose each component $P^j_{d,t}$ is a polynomial in $x$ of degree less than or equal to $d$ such that 
  \begin{align}
   |u_j (x,t)- P^j_{d,t} (x)| \leq C |(x,t)|^{d+\alpha}\comma (x,t) \in Q_{1/2},
   \label{EQ54}
\end{align}
for $j=1,\ldots, n$, where $\alpha \in (0,1)$ and $C$ is a positive constant depending on $n$, $d$, $q$, $\alpha$, and $u$. Moreover, $P$ satisfies the Stokes system
  \begin{align}
   &\partial_t P^j_d - \triangle P^j_d   + \partial_j R (x) = \sum_{i=d-1}^{d} \sum_{|\beta|=i} C_{\beta, t} \frac{x^{\beta}}{\beta !} 
   \comma j=1,\ldots,n
   \label{EQ55}
    \\&\nabla\cdot P  =0,
   \label{EQ56}
\end{align}
where $R$ a suitable pressure term depending on $u$ and $p$. 
\end{Corollary}
\colb

\startnewsection{The basic results}{sec3}
We start by recalling pointwise estimates on the derivatives of solutions to the
homogeneous heat equation
  \begin{align}
   \partial_t u - \triangle u = 0
   \period
   \label{EQ32}
\end{align}
The fundamental solution is given by
$
\Gamma (x,t) = {{(4\pi t)}^{-n/2}} \exp \left( - \fractext{|x|^2}{4t}\right)
$ for $t>0$ and 
$\Gamma(x,t) =0$ for $t \leq 0$.  
Recall that the derivatives are bounded as
%
  \begin{align}
   |\partial_t^l \partial_{x}^{\mu} \Gamma (x,t) | \leq \frac{C(\mu,l)}{(|x|+ \sqrt{t})^{n+|\mu|+2l}} e^{-|x|^2 /8t} 
   \comma l\in{\mathbb N_0}
   \commaone \mu\in{\mathbb N}_{0}^{n}
   \label{EQ17}
   \period
  \end{align}
First recall that for any solution $u$ 
of \eqref{EQ32} we have
  \begin{align}
   |D_x^{\mu}D_t^l u(x,t)| \leq \frac{C}{(R- |(x,t)|)^{|\mu|+2l}} \sup_{Q_R} |u|
   \comma (x,t) \in Q_{R/2}
  \label{EQ18} 
  \end{align}
where C depends on $ |\mu|+2l$ \cite{Li}.


For completeness, we briefly recall the derivation of the
fundamental solution to the Stokes system \eqref{EQ11}--\eqref{EQ12}.
Let $u(0,\cdot)=u_0$ be the initial condition. 
By uniqueness of solutions, we have
  \begin{align}
   u_k (x,t) & = \int_{\mathbb{R}^n} \Gamma (x-y,t) u_0 (y) \,dy + \int_0^t \int_{\mathbb{R}^n} \Gamma (x-y,t-s) f_k (y,s) \,dy \,ds  
      \nonumber \\& \indeq
        - \int_0^t \int_{\mathbb{R}^n} \Gamma (x-y,t-s) \partial_k p (y,s) \,dy\,ds,
   \label{EQ33}
  \end{align}
for $k=1,\ldots,n$. Using the Fourier transform of both sides in \eqref{EQ14}, we get
  \begin{align*}
   \partial_k p = -R_j R_k f_j 
     \comma k=1,\ldots,n,
  \end{align*} 
where 
  \begin{align*}
  R_j g = \left(\frac{\xi_j}{i|\xi|} \hat{g}
          \right)^{\check{}}
  \end{align*}
denotes the $j$-th Riesz transform, using the
Fourier transform
$\hat f(\xi)=\int f(x)e^{- i \xi \cdot x}\,dx$.
Thus \eqref{EQ33} can be written as
  \begin{align}
   u_k (x,t) = \int_{\mathbb{R}^n} \Gamma (x-y,t) u_0 (y) \,dy + \int_0^t \int_{\mathbb{R}^n} K_{jk} (x-y,t-s) f_j (y,s) \,dy \,ds,
   \label{EQ34}
  \end{align} 
where 
  \begin{equation}
   K_{jk} (x,t)
    = \delta_{jk} \Gamma(x,t) 
        + R_j R_k \Gamma (x,t)
   \comma j,k=1,\ldots,n   
   \period
   \label{EQ47}
  \end{equation}
For each $j,k=1,\ldots, n$, the function $K_{jk}$ solves the heat equation, i.e.,
  \begin{align}
   \partial_t K_{jk} (x,t) - \triangle K_{jk} (x,t) =0 
   \comma t>0
   \period
   \label{EQ35}
\end{align}
Also,
  \begin{align}
   \partial_j K_{jk} (x,t) =0
   \comma j=1,\ldots,n   
   \period
   \label{EQ36}
\end{align}
Furthermore, we have the estimate
  \begin{align}
   |D_x^{\mu}D_t^l K_{jk} (x,t)| \leq \frac{C}{|(x,t)|^{n+ |\mu| + 2l}}
   \comma l\in{\mathbb N}_0
   \comma \mu\in{\mathbb N}_{0}^{n}
   \label{EQ25}
  \end{align}
where $C$ depends on  $|\mu|$ and $l$
\cite{FJR,L,S}.

\startnewsection{Proof of the Main Theorem}{sec4}
In the next lemma, we
construct a solution of the system 
\eqref{EQ11}--\eqref{EQ12} which vanishes  with a certain prescribed
degree. 

\cole
\begin{Lemma}
\label{L01}
Assume that $f=(f_1,\ldots,f_n) \in L^q {(Q_1)}$, 
where $q > 1+ n/2$, satisfies  \eqref{EQ15}. 
Then there exists $u=(u_1, \ldots, u_n)\in W_q^{2,1} (Q_1)$
which solves \eqref{EQ11}--\eqref{EQ12} 
and satisfies
  \begin{align}
   \|u\|_{L^q (Q_1)} ,  \|u\|_{W^{2,1}_q (Q_{3/4})} \leq C \gamma 
   \period
   \label{EQ10}
  \end{align}
Furthermore,
  \begin{align}
   |u_k(x,t)| \leq C \gamma |(x,t)|^{d+ \alpha}
   \comma  |(x,t)|< 1/2
   \commaone k=1, \ldots,n
   \label{EQ07}
  \end{align}
where the constant $C$ depends on $n$, $d$, and $\alpha$.  
\end{Lemma}
\colb

\begin{proof}[Proof of Lemma~\ref{L01}]
We start by setting
  \begin{align}
   w_k (x,t) 
     &= \int K_{jk} (x-y, t-s) f_j (y,s) \,dy \,ds
     \nonumber\\&
     = \int_{|(y,s)|<1} K_{jk} (x-y, t-s) f_j (y,s) \,dy \,ds
   \comma k=1,\ldots,n
   \label{EQ21}
  \end{align}
and
  \begin{equation}
   p
   = \Delta^{-1}\nabla\cdot f
   =(\xi_j \hat f_j/i|\xi|^2)\check{\;}
   \period
   \label{EQ49}
  \end{equation}
Then we have 
  \begin{align}
   \partial_t w_k - \triangle w_k 
    {+ \partial_{k} p}
   = f_k
   \comma k=1,\ldots,n
  \label{EQ22}
 \end{align}
and
   \begin{align}
    \|w\|_{W_q^{2,1} (Q_1)} & \leq C \|f \|_{L^q (Q_1)}
    \leq C \gamma
   \period
   \label{EQ23}
   \end{align}
Now, we consider the Taylor expansion of $K_{jk} (x-y, t-s)$ around $(0,0)$. 
Let 
$|(y,s)|<1$ be such that $s\neq0$.
Denote by $K_{jk}^m$  the $m$-th order terms, i.e.,
  \begin{align}
   K_{jk}^m (x,y; t,s) = \sum_{|\mu|+ 2l = m} D_x^{\mu} D_t^l K_{jk} (-y,-s) \frac{x^{\mu} t^l}{\mu ! l!}\period  
   \label{EQ24}
   \end{align}
It is easy to check that $K_{jk}^m$ solves 
the heat equation for each $j,k$, i.e.,
 \begin{align*}
 \partial_t K_{jk}^m (x,y; t,s) - \triangle_x K_{jk}^m (x,y; t,s) = 0 
 \end{align*}
 for any $(y,s)$. 
Let
  \begin{align}
   v_k (x,t) = \int_{|(y,s)|<1} \sum_{m=0}^{d} K_{jk}^m (x,y; t,s) f_j (y,s) \,dy \,ds
   \comma k=1,\ldots,n
   \period
   \label{EQ26}
  \end{align}
Each $v_k$ is a polynomial of degree less than or equal to $d$ and
it satisfies
  \begin{align}
   \partial_t v_k (x,t) - \triangle v_k (x,t) =0
   \period
  \label{EQ27}
  \end{align}
Moreover, we have 
  \begin{equation}
   \nabla\cdot v=0   
   \label{EQ50}
  \end{equation}
Indeed,
we may write
  \begin{align*}
   \partial_k v_k = \int_{|(y,s)|<1} \sum_{m=1}^d \sum_{\substack{|\mu|+ 2l =m \\ \mu_k >0}} D_x^{\mu} D_t^l K_{jk} (-y, -s) \frac{x^{  \mu-e_k  }t^l}{(\mu-e_k)! l!} f_j (y,s) \,dy \,ds,
  \end{align*}
where $e_k$ is the standard $k$-th unit vector in ${\mathbb R}^{n}$.
Note that
  \begin{align*}
   \sum_{\substack{|\mu|+ 2l =m \\ \mu_k >0}} D_x^{\mu} D_t^l K_{jk} (-y, -s) \frac{x^{{\mu-e_k}}t^l}{(\mu-e_k)! l!}  = \sum_{|\bar\mu| + 2l = m-1} \partial_k D_x^{\bar{\mu}} D_t^l K_{jk} (-y, -s) \frac{x^{\bar{\mu}}t^l}{\bar{\mu}! l!}
  \period
  \end{align*}
Using 
$\partial_{k} K_{jk}=0 $
for $j=1,\ldots,n$, we get $\nabla\cdot v=0$.
Now, set 
  \begin{align}
   u_k (x,t) & = w_k (x,t) -  v_k (x,t) 
  \nonumber\\ 
  \label{EQ29}
  & = \int_{|(y,s)| <1} 
        \left( 
              K_{jk} (x-y, t-s) - \sum_{m=0}^d K_{jk}^m (x,y;t,s) 
        \right) f_j (y,s) \,dy \,ds
   \comma k=1,\ldots,n
  \end{align}
and note that we have
  \begin{align}
   &\partial_t u_k 
        - \triangle u_k 
        + \partial_k p  
     = f_k 
   \comma k=1,\ldots,n
   \period
\end{align}
We now check the condition \eqref{EQ12}.  
Since
  \begin{align*}
   \partial_k w_k = \int_{|(y,s)|<1} \partial_k K_{jk} (x-y, t-s) f_j (y,s) \,dy \,ds
   =0
   \period
  \end{align*}
where we used $\partial_{k} K_{jk}=0$ for $j=1,\ldots,n$,
we get $\nabla \cdot u=0$.

Next, we claim that 
  \begin{align}
   |u(x,t)| \leq C \gamma |(x,t)|^{d+\alpha} 
   \comma  |(x,t)| \leq \frac{1}{2}\period 
  \label{EQ30}
  \end{align}
Fixing $|(x,t)|\le 1/2$, we split the integral 
on the far right side of \eqref{EQ29}
into three parts
  \begin{align*}
    I_1 & = \int_{|(y,s)|\le 2 |(x,t)|} K_{jk} (x-y, t-s) f_j (y,s) \,dy \,ds \\
    I_2 & = - \int_{|(y,s)|\le 2 (x,t)} \sum_{m=0}^{d} K_{jk}^m (x,y; t,s) f_j (y,s) \,dy \,ds \\
    I_3 & = \int_{2 |(x,t)|< |(y,s)|<1} \left(K_{jk} (x-y, t-s) - \sum_{m=0}^d K_{jk}^m (x,y; t,s)\right) f_j (y,s) \,dy \,ds\period 
  \end{align*}
By a hypothesis, $q > 1 + \fractext{n}{2}$. Therefore, 
by  H\"older's inequality and \eqref{EQ25},
  \begin{align*}
   |I_1| & \leq C\sum_{j=1}^n \left( \int_{|(y,s)|\le 2 |(x,t)|} \frac{dy \,ds}{|(x-y, t-s)|^{nq'}} \right)^{1/q'} \left( \int_{|(y,s)|\le 2|(x,t)|} |f_j (y,s)|^q \,dy \,ds \right)^{1/q} \\
   & \leq C\sum_{j=1}^n \left( \int_{|(y,s)|< 3|(x,t)|} \frac{dy\,ds}{|(y,s)|^{nq'}}  \right)^{1/q'} \left( \int_{|(y,s)|< 2|(x,t)|} |f_j (y,s)|^q \,dy\,ds\right)^{1/q} \\
   & \leq C  \gamma |(x,t)|^{\fractext{(n+2)}{q'} -n} |(x,t)|^{d-2 + \alpha + \fractext{(n+2)}{q}} \\
   & = C \gamma |(x,t)|^{d+ \alpha},
  \end{align*}
where $q'=(q-1)/q$.
Similarly, using \eqref{EQ25} we estimate
  \begin{align}
   |I_2| & \leq C \sum_{j=1}^n \sum_{k=0}^d |(x,t)|^k \int_{|(y,s)| < 2 |(x,t)|} \frac{|f_j (y,s)|}{|(y,s)|^{n+k}} \,dy \,ds \nonumber\\
   & \leq C \sum_{j=1}^n \sum_{k=0}^d |(x,t)|^k  \sum_{i=0}^{\infty} \int_{\fractext{|(x,t)|}{2^i} < |(y,s)| < \fractext{|(x,t)|}{2^{i-1}}} \frac{|f_j (y,s)|}{|(y,s)|^{n+k}} \,dy \,ds \nonumber\\
   & \leq C  \gamma \sum_{k=0}^d |(x,t)|^k \sum_{i=0}^{\infty} \left( \frac{|(x,t)|}{2^i}\right)^{d-k +\alpha} \nonumber\\
   & \leq C  \gamma |(x,t)|^{d+\alpha}
  \period
  \end{align}
In order to estimate $I_3$, we expand $K_{jk} (x-y, t-s)$ into Taylor series around $(x,t) = (0,0)$. For each $j,k = 1, \ldots, n$, we have
  \begin{align}
   K_{jk} (x-y, t-s) = \sum_{i=0}^d \sum_{|\mu| + l =i} D^{\mu,l} K_{jk} (-y, -s) \frac{x^{\mu}t^l}{\mu ! l!} + \sum_{|\mu|+l= d+1} D^{\mu, l} K_{jk} (\xi x -y, \eta t-s) \frac{x^{\mu}t^l}{\mu! l!}, 
   \label{EQ31}
  \end{align}
where $0< \xi = \xi (x,t; y,s)$, 
$\eta = \eta (x,t; y,s) < 1$. Therefore,
  \begin{align*}
    &K_{jk} (x-y, t-s) - \sum_{m=0}^d K_{jk}^m (x,y; t,s) \\ 
    & \indeq = \sum_{i=0}^d \sum_{\substack{|\mu| +l=i \\ |\mu|+2l \geq d+1}} 
     D^{\mu, l} K_{jk}(-y,-s) \frac{x^{\mu} t^l}{\mu ! l! } 
      + \sum_{|\mu|+ l = d+1} D^{\mu,l} K_{jk} (\xi x -y, \eta t-s) \frac{x^{\mu} t^l}{\mu ! l!}\period
  \end{align*}
Using  the bound on $|\partial_x^{\mu} \partial_t^l K_{jk} (x,t) |$, the difference above can be estimated with
  \begin{align}
   C \sum_{ i= d+1}^{2(d+1)} \left( \frac{1}{|(y,s)|^{n+i}} + \frac{1}{|(\xi x -y, \eta t -s)|^{n+i}}\right) |(x,t)|^i\period
  \end{align}
As we assumed $2|(x,t)|< |(y,s)|$, we get 
  \begin{align*}
   2|(\xi x -y, \eta t -s)|> |(y,s)|\period
  \end{align*}
Now we may estimate
  \begin{align*}
   |I_3| & \leq \int_{2|(x,t)|< |(y,s)|<1} 
     \left|K_{jk} (x-y, t-s) - \sum_{m=0}^d K_{jk}^m (x,y; t,s)\right| |f_j(y,s)| \,dy \,ds, \\
   & \leq C \sum_{i= d+1}^{2(d+1)}|(x,t)|^i  \int_{2|(x,t)|< |(y,s)|< 1}  \sum_{j=1}^n \frac{|f_j (y,s)|}{|(y,s)|^{n+i}} \,dy \,ds, \\
& \leq  C\sum_{i= d+1}^{2(d+1)} |(x,t)|^i \sum_{u=1}^M \int_{2^u |(x,t)|< |(y,s)|< 2^{u+1} |(x,t)|} \sum_{j=1}^n \frac{|f_j (y,s)|}{|(y,s)|^{n+i}} \,dy \,ds, \\
& \leq C\sum_{i= d+1}^{2(d+1)} |(x,t)|^i \sum_{u=1}^M (2^u |(x,t)|)^{-(n+i)} 
   \sum_{j=1}^n \|f_j \|_{L^q (Q_{{2^{u+1} |(x,t)|}})}  
   \int_{2^u |(x,t)|< |(y,s)|< 2^{u+1}}  \,dy \,ds 
  \end{align*}
from where
  \begin{align*}
   |I_3| & \leq   \sum_{i= d+1}^{2(d+1)} |(x,t)|^i \sum_{u=1}^M C  \gamma (2^u |(x,t)|)^{d+\alpha -i}, \\
   & \leq C  \gamma |(x,t)|^{d+\alpha} \sum_{i= d+1}^{2(d+1)} \sum_{u=1}^M \frac{1}{(2^{i-d-\alpha})^u}, \\ 
   &\leq C  \gamma |(x,t)|^{d+\alpha},
  \end{align*}
which proves \eqref{EQ07}. For \eqref{EQ10} recall that we have the same estimate on 
$\|w_k\|_{L^q (Q_1)}$ by \eqref{EQ23}. Furthermore, we may show that for $k=1,\ldots,n$ 
  \begin{align}
   |v (x,t)| \leq C \gamma
   \comma (x,t) \in Q_1, 
    \label{EQ38}
  \end{align}
by following the same approach we took in estimating $I_2$. Combining \eqref{EQ38} 
with \eqref{EQ18} and \eqref{EQ27}, we get
  \begin{align}
   \|v\|_{W^{2,1}_q (Q_{r})} \leq C(r) \gamma 
   \comma r<1,
   \label{EQ39}
  \end{align}
which completes the proof of \eqref{EQ10}.
\end{proof}

\begin{proof}[Proof of Theorem~\ref{T01}]
Suppose $(u,p)$ solves \eqref{EQ11}--\eqref{EQ12}.
In Lemma~\ref{L01}, we have already constructed  a solution $(\tilde{u}, \tilde{p})$ of \eqref{EQ11}--\eqref{EQ12} 
with $\tilde{u}_k \in W_q^{2,1} (Q_1)$ for each $k=1, \ldots, n$, such that
  \begin{align*}
   |\tilde{u}_k (x,t)| \leq C \gamma |(x,t)|^{d+\alpha} 
   \comma (x,t) \in Q_{1/2} 
  \end{align*}
for each $k= 1, \ldots, n$, where $C$ is a constant depending on $n$, $d$, and $\alpha$. Also, we have 
  \begin{align}
   \|\tilde{u}_k \|_{L^q (Q_1)},
   \|\tilde{u}_k \|_{W^{2,1}_q (Q_{1/2})}
   \leq C \gamma
  \comma k=1,\ldots,n 
   \period
   \label{EQ09}
  \end{align}
Then we set $U = u- \tilde{u}$. Note that $U$ solves the system
  \begin{align}
   & U_t - \triangle U + \nabla (p- \tilde{p}) = 0, 
   \label{EQ40}\\
   & \nabla \cdot U =0
   \period
   \label{EQ41}
  \end{align}
Furthermore, we consider the vorticity equation. Let $W = \nabla \times U$ 
denote the curl of $U$, i.e.,
$W_{i,j}=\partial_{i} U_j - \partial_{j}U_i$. 
Then $W= [W_{i,j}]_{n \times n}$ satisfies the heat equation
  \begin{align*}
    W_t - \triangle W = 0\period 
  \end{align*} 
As $W_{i,j} = \partial_i (u_j - \tilde{u}_j) - \partial_j (u_i - \tilde{u}_i) $, we obtain
  \begin{align}
    \sum_{i,j=1}^n \|W_{i,j}\|_{L^q (Q_r)} & \leq \sum_{k=1}^n \| u_k\|_{W^{2,1}_q(Q_r)} + \sum_{k=1}^n \| \tilde{u}_k \|_{W^{2,1}_q(Q_r)} 
     \leq C \gamma + \sum_{k =1}^n \|u_k\|_{W^{2,1}_q (Q_1)} 
   \label{EQ42}
\end{align} 
for any $r<1$.
By expanding $U_k$ into Taylor series in $x$, we obtain 
  \begin{align}
   U_k (x,t) = P^k_{d,t} (x) + R^k_{d,t} (x) \comma k=1, \ldots, n,
   \label{EQ43}
  \end{align}
where
  \begin{align}
   P^k_{d,t}(x) & = \sum_{i=0}^d \sum_{|\alpha|=i} D^{\alpha}_x U_k (0,t) \frac{x^{\alpha}}{\alpha !} 
   \label{EQ44}
  \end{align}
and
  \begin{align}
   R^k_{d,t} (x) & = \sum_{|\alpha| = d+1} D^{\alpha}_x U_k (\xi (x,t) x,t) \frac{x^{\alpha}}{\alpha !} 
   \label{EQ45}
\end{align}
where $0< \xi(x,t) \leq 1$. Let $(x,t) \in Q_{1/2}$. Note that
  \begin{align*}
   |R^k_{d,t} (x)| & \leq C \sum_{|\alpha| = d+1} |D^{\alpha}_x U_k (\xi(x,t) x ,t)| |x^{\alpha}| \\
   & \leq C |x|^{d+1} \sum_{|\alpha|=d+1} |D^{\alpha}_x U_k (\xi (x,t)x,t)| \\
   & \leq C |(x,t)|^{d+1} \sum_{|\alpha|= d+1} |D^{\alpha}_x U_k (\xi (x,t) x,t) |
   \period
  \end{align*} 
Also, selecting $1/2 < r_1 < r_2 < 3/4$, we get
  \begin{align*}
    & |D^{\alpha}_x U_k (\xi (x,t)x,t)| \leq \|D^{\alpha}_x U_k (\cdot, t) 
                                         \|_{L^{\infty}_x (B_{1/2} (0,0))} \\ 
   & \indeq \leq C \| D^{\alpha}_x U_k (\cdot, t)\|^{1/2}_{L^q_x (B_{1/2} (0,0))} \| D^{\alpha +n}_x U_k (\cdot, t)\|^{1/2}_{L^q_x (B_{1/2} (0,0))} + C \| U_k (\cdot, t)\|_{L^q_x (B_{1/2} (0,0))} 
\\ 
   & \indeq \leq C \left( \|  U_k (\cdot,t)\|_{L^q_x (B_{1/2} (0,0))}  +  \| D^{\alpha} U_k (\cdot, t)\|_{L^q_x (B_{1/2} (0,0))}   + \| D^{\alpha+n} U_k (\cdot,t)\|_{L^q_x (B_{1/2} (0,0))}  \right) \\ 
   & \indeq \leq C \left( \|  u_k (\cdot,t) - \tilde{u}_k (\cdot,t)\|_{L^q_x (B_{1/2} (0,0))} + \sum_{i,j = 1}^n  \| D^{\alpha -1} W_{i,j} (\cdot, t) \|_{L^q_x (B_{r_1} (0,0))}  \right. \\
   & \indeq\indeq\indeq\indeq\indeq \left. + \sum_{i,j =1}^n \| D^{\alpha +n -1} W_{i,j} (\cdot, t)\|_{L^q_x (B_{r_1} (0,0))}  \right)\period 
  \end{align*}
We may further bound 
  \begin{align*}
   \|D^{\beta} W_{i,j}\|_{L^{\infty}_t  [0, 1/2] L^q_x (B_{r_1} (0,0))}  \leq C \|W_{i,j}\|_{L^q_t  [0, 3/4] L^q_x (B_{3/4}(0,0))}
  \end{align*}
for any multiindex $\beta$. By \eqref{EQ09} and \eqref{EQ42} we have the desired bound on $R^k_{d,t}$. Going back to \eqref{EQ43}, we obtain the extension
  \begin{align*}
   |u_k (x,t) - P^k_{d,t} (x)| & \leq |\tilde{u}_k (x,t)| + |R^k_{d,t} (x,t)| \\
   \indeq & \leq C \gamma |(x,t)|^{d+\alpha} + C \left(\gamma + \sum_{k=1}^n \|u_k\|_{W^{2,1}_q (Q_1)} \right) |(x,t)|^{d+1},
  \end{align*}
for any $|(x,t)| \leq 1/2$.
Furthermore, $P$ satisfies \eqref{EQ16} and \eqref{EQ01}.  
Taking the divergence of \eqref{EQ44} we get
  \begin{align*}
   \partial_k P_{d,t} ^k & = \sum_{i=0} ^{d-1} \sum_{|\alpha|= i-1} \partial_k  D^{\alpha} _x U_k (0,t) \frac{x^{\alpha}}{\alpha !} 
    =0,
  \end{align*}
as $U = u- \tilde{u}$ is divergence free.  Similarly,
  \begin{align*}
   & \partial_t P^k _{d,t} (x) - \triangle P^k _{d,t} (x)  \\ 
   & \indeq = \sum_{i=0} ^d \sum_{|\alpha|=i} D^1 _t D^{\alpha} _x U_k (0,t) \frac{x^{\alpha}}{\alpha !} - \sum_{i=0} ^{d-2} \sum_{|\alpha|=i} \partial_k ^2 D^{\alpha} _x U_k (0,t) \frac{x^{\alpha}}{\alpha !}  \\
  & \indeq = \left( \sum_{i=0} ^{d-2} \sum_{|\alpha|=i} D^{\alpha} _x (\partial_k p - \partial_k \tilde{p}) \right) + \sum_{i={d-1}} ^d \sum_{|\alpha|=i} D^1 _t D^{\alpha} _x U_k (0,t) \frac{x^{\alpha}}{\alpha !},
  \end{align*}
which proves \eqref{EQ16}.
\end{proof}

\begin{proof}[Proof of Theorem~\ref{T01}]
The proof of this result follows that of Theorem~\ref{T01} and it is
thus omitted.
\end{proof}

\nnewpage
\section*{Acknowledgments} 
The authors were supported in part by the NSF grant DMS-1311943.

\end{document}